\documentclass[11pt]{amsart}
\usepackage{amssymb}
\usepackage{amsmath}
\usepackage{amsthm}
\usepackage[latin1]{inputenc}
\usepackage[T1]{fontenc}
\usepackage{hyperref}
\usepackage[all]{xy}
\usepackage{enumerate}

\newcommand{\EX}{\ensuremath{\mathcal{E}(X)}}
\newcommand{\EH}{\ensuremath{\mathcal{E}_H(X)}}
\newcommand{\EG}{\ensuremath{\mathcal{E}(G)}}
\newcommand{\EP}{\ensuremath{\mathcal{E}_P(G)}}
\newcommand{\EL}{\ensuremath{\mathcal{E}_L(G/N)}}
\newcommand{\E}{\ensuremath{\mathcal{E}(G/N)}}
\newcommand{\Eio}{\ensuremath{\mathcal{E}_i^o}}
\newcommand{\Ei}{\ensuremath{\mathcal{E}_i}}
\def\g{\mathfrak{g}}
\def\k{\mathfrak{k}}
\def\p{\mathfrak{p}}
\def\a{\mathfrak{a}}
\def\ap{\mathfrak{a}_\p}
\def\np{\mathfrak{n}_\p}
\def\nk{\mathfrak{n}}
\def\nb{\bar{\mathfrak{n}}}

\def\Ap{A_\p}
\def\Mp{M_\p}
\def\M'p{M'_\p}
\def\Np{N_\p}
\def\Md{\widehat{M}_d}
\def\bn{\bar{n}}
\def\bnu{\bar{\nu}}

\newcommand{\n}{\boldsymbol{\bar{n}}}
\newcommand{\m}{\boldsymbol{m}}
\newcommand{\ab}{\boldsymbol{a}}
\newcommand{\lm}{\boldsymbol{l}}

\def\Iw{\mathcal{I}_w}

\newcommand{\test}{C_c^\infty}

\DeclareMathOperator{\Ad}{Ad}
\DeclareMathOperator{\Ker}{Ker}

\newcommand{\N}{\ensuremath{\bar{N}}}

\newcommand{\R}{\mathbb{R}}
\newcommand{\C}{\mathbb{C}}
\newcommand{\GR}{\ensuremath{\mathrm{SL}_2(\R)}}
\newcommand{\GC}{\ensuremath{\mathrm{SL}_2(\C)}}

\newcommand{\scal}[2]{\langle #1,#2\rangle}
\newcommand{\ind}[3]{\ensuremath{\mathop{\rm{Ind}}\nolimits_{#1}^{#2}{#3}}}
\newcommand{\supp}[1]{\ensuremath{\mathop{\rm{Supp}}#1}}

\theoremstyle{plain}
\newtheorem{theorem}{Theorem}
\newtheorem{lemma}{Lemma}
\newtheorem{proposition}{Proposition}
\newtheorem{corollary}{Corollary}
\theoremstyle{definition}
\newtheorem{definition}{Definition}
\newtheorem{example}{Example}
\newtheorem{notation}{Notation}
\newtheorem{remark}{Remark}

\begin{document}

\title[Hilbert modules for parabolically induced representations]{Hilbert modules associated to parabolically induced representations}
\author{Pierre Clare}
\address{Pierre Clare\\ The Pennsylvania State University\\ Department of Mathematics\\ McAllister Building\\ University Park, PA - 16802}
\email{clare@math.psu.edu}
\date{}

\setcounter{tocdepth}{1}

\begin{abstract}
To a measured space carrying two group actions, we associate a Hilbert $C^*$-module in a way that generalises Rieffel's construction of induction modules. This construction is then applied to describe the $P$-series of a semisimple Lie group. We provide several realisations of this module, corresponding to the classical pictures for the $P$-series. We also characterise a class of bounded operators on the module which satisfy some commutation relation, and interpret the result as a generic irreducibility theorem. Finally, we establish the convergence of standard intertwining integrals on a dense subset of this module.
\end{abstract}

\maketitle

\section{Introduction}\label{intro}

\subsection{Motivation and outline} In \cite{RieffelIRC*A}, M. A. Rieffel developed a general theory of induced representations for $C^*$-algebras by means of Hilbert modules. Namely, in the special case of group $C^*$-algebras, given a closed subgroup $H$ of a locally compact group $G$, his construction yields a $C^*(H)$-Hilbert module $E_H^G$ equipped with an action of $C^*(G)$ by bounded operators. This module \textit{contains} all representations induced from $H$ to $G$, in the sense that, for every representation $(\rho,\mathcal{H}_\rho)$ of $H$, there exists a map between $E_H^G\otimes_{C^*(H)}\mathcal{H}_\rho$ and the space of $\ind{H}{G}{\rho}$, that preserves the scalar products on these Hilbert spaces, and intertwines the actions of $C^*(G)$. Among the advantages of this point of view is the possibility to express Mackey's Imprimitivity Theorem under the neat form of the Morita equivalence between $C^*(H)$ and $C_0(G/H)\rtimes G$.

Understanding the reduced dual $\hat{G}_r$ of a group $G$ as a measured space amounts to explicitely writing down its Plancherel measure. This task was achieved in the case of semisimple Lie groups by Harish-Chandra in \cite{HC70}. In particular, it is shown that besides the atomic part, that is the discrete series, the support of the Plancherel measure consists in a special type of induced representations, namely the $P$-series. Let us fix notations: $G$ is a linear connected semisimple Lie group with finite center, and $P$ a cuspidal parabolic subgroup with Langlands decomposition $P=MAN$. We denote by $\Md$ the discrete series of $M$ and $\widehat{A}$ the unitary dual of $A$.

\begin{definition}\label{Pseries}
The \textit{$P$-series} of $G$ is the family of representations of the form \[\ind{P}{G}{\sigma\otimes\chi\otimes\rm{1}},\] where $\sigma\in\Md$ and $\chi\in\widehat{A}$.
\end{definition}

Among the facts involved in the description of $\hat{G}_r$, which will be explained in greater detail below, it is shown that these representations are generically irreducible \cite{HC71}, and that the families $\widehat{G}_P$ of classes of irreducible components of the $P$-series representations are in fact parametrised by the conjugacy classes of Levi components $L=MA$ of the cuspidal parabolic subgroups $P$. In view of these facts, it appears that properly describing $P$-series representations within the Hilbert modules setting should involve modules over $C^*(L)$ rather than over $C^*(P)$, as would a direct application of Rieffel's theory. 

Following this remark, we were led to generalise the construction of \cite{RieffelIRC*A} in order to obtain Hilbert modules suited to the description of $P$-series. Section \ref{construction} presents our construction, in which the data consist in two groups $G$ and $H$ acting on a space $X$ equipped with a measure satisfying certain equivariance hypotheses. The outcome is a $C^*(H)$-Hilbert module $\EX$ carrying an action of $C^*(G)$ by bounded operators. We apply the procedure to the special case of $P$-series in Section \ref{E(G/N)}: it yields Hilbert modules over the $C^*$-algebra of the Levi component, which nevertheless still induce the $P$-series. Section \ref{pictures} is devoted to providing different pictures of these modules, corresponding to classical properties of $P$-series representations. In Section \ref{irred}, we characterise bounded invariant operators commuting to the action of $C^*(G)$ and interpret the result as a generic irreducibility result, reflecting theorems of Bruhat and Harish-Chandra in the $C^*$-algebraic framework. Another application of this point of view is given in Section \ref{intertwine}, where we establish convergence of intertwining integrals.

As explained at the end of the paper, the modules discussed here also constitute the framework for a global theory of Knapp-Stein intertwining operators. In the classical representation theory, these operators allow to account for subtle reducibility phenomena, that is fine topological properties of the reduced dual, which do not manifest in the study of the Plancherel measure. Following the general principle of Noncommutative Geometry, the reduced dual of a group should be studied \textit{via} its $C^*$-algebra. Subsequentely, developing a theory of intertwiners at the Hilbert module level is likely to provide a useful tool in the analysis of reduced $C^*$-algebras associated to Lie groups, as advocated in Section \ref{images}.

\subsection{General notations and preliminaries}\label{generalnotations}
The reader is referred to \cite{HelgasonDS,HelgasonGASS}, \cite{Knapp1,KnappBeyond}, and \cite{Lipsman} for general facts about structure and representation theory of semisimple Lie groups, and to \cite{Lance} for the theory of Hilbert modules.

The space of compactly supported functions on a topological space $X$ with complex values is denoted by $C_c(X)$. If $G$ is a locally compact group, we write $dg$ for a left Haar measure on $G$, and $\Delta_G$ for the corresponding modular function. The maximal and reduced $C^*$-algebras of the group are respectively denoted by $C^*(G)$ and $C^*_r(G)$. If $A$ is a $C^*$-algebra and $E,F$ are $A$-Hilbert modules, the set of bounded (that is adjointable) operators between $E$ and $F$ is denoted by $\mathcal{L}_A(E,F)$ or simply $\mathcal{L}(E,F)$, while the compact operators are denoted by $\mathcal{K}_A(E,F)$ or $\mathcal{K}(E,F)$. In the case $E=F$, the set $\mathcal{L}(E)$ is a $C^*$-algebra containing $\mathcal{K}(E)$ as a two-sided ideal.

The \textit{multipliers of $E$} are the elements of $\mathcal{L}(A,E)$, also denoted $\mathcal{M}(E)$, where $A$ is viewed as a Hilbert module over itself. The module $E$ may be embedded in $\mathcal{M}(E)$ by associating to $\xi\in E$ the map $m_\xi:a\longmapsto\xi.a$ on $A$, with adjoint map $\scal{\xi}{\cdot}$. If $E=A$, we recover one of the classical equivalent definitions of the \textit{multiplier algebra} of $A$, and $\mathcal{M}(E)$ is a Hilbert module over $\mathcal{M}(A)$. The construction giving $\mathcal{M}(E)$ is functorial and for $T\in\mathcal{L}(E,F)$, we write $M(T):\mathcal{M}(E)\rightarrow\mathcal{M}(F)$ the operator of left composition by $T$, with adjoint $M(T^*)$. For $\xi\in E$ and $\eta\in F$, \[m_{\scal{T.\xi}{\eta}} = \scal{M(T).m_\xi}{m_\eta}\] holds with inner products respectively taking values in $F$ and $\mathcal{M}(F)$.

An example of elements in $\mathcal{M}(C^*(G))$ is given by extending left translations of compactly supported functions on $G$: for $g\in G$, we denote by $U_g$ the multiplier of $C^*(G)$ defined by $U_g.f=f(g^{-1}\cdot)$ for any $f\in C_c(G)$.

\section{General construction}\label{construction}

Let $X$ be a locally compact topological space, with commuting actions of two locally compact groups $G$ and $H$, from the left and the right respectively. The $H$-action is assumed to be proper, so that $X/H$ is locally compact. In what follows, we will also need a paracompactness assumption on $X/H$, which will always be satisfied in the further examples, where $X/H$ will turn out to be a compact manifold. Let us finally assume that $X$ carries a $G$-invariant Borel measure $\mu$, which is $H$-relatively invariant with character $\delta_X$. It means that the relation $d\mu(g.x.h) = \delta_X(h)d\mu(x)$ is satisfied for any $g\in G$, $h\in H$.
From the above data \[G\curvearrowright (X,\mu)\curvearrowleft H,\] we build a right $C^*(H)$-Hilbert module together with a left action of $C^*(G)$ by bounded operators. This module will be obtained as the completion of $C_c(X)$ with respect to an appropriate norm. Let us first describe the pre-Hilbert module structure on $C_c(X)$ over the dense  involutive subalgebra $C_c(H)$ of $C^*(H)$. For any $f\in C_c(X)$ and $\varphi\in C_c(H)$, define \[(f.\varphi)(x)=\int_H\frac{1}{\Delta_H(h)^{\frac{1}{2}}\delta_X(h)^{\frac{1}{2}}}f(x.h^{-1})\varphi(h)\,dh\] for all $x\in X$.

\begin{notation}
The map $H\rightarrow\R_+^*$ defined by $\delta_X^{\frac{1}{2}}\Delta_H^{-\frac{1}{2}}$ will be denoted $\gamma_X$.
\end{notation}

\begin{proposition}\label{scal}
For $f,g\in C_c(X)$ and $h\in H$, let \[\scal{f}{g}(h)=\gamma_{X}(h)\int_X\overline{f(x)}g(x.h)\,d\mu(x).\] The map thus defined on $C_c(X)\times C_c(X)$ is a $C_c(H)$-valued inner product on $C_c(X)$.
\end{proposition}

\begin{proof}
Let $f,g\in C_c(X)$ and $\varphi\in C_c(H)$. The relations $\scal{f}{g}^*=\scal{g}{f}$ and $\scal{f}{g.\varphi}=\scal{f}{g}\varphi$ in $C^*(H)$ follow from straightforward calculations and the definition of $\gamma_{X}$. The positivity of $\scal{f}{f}$ in $C^*(H)$ relies on the use of a \textit{Bruhat section} for $X$ over $X/H$: it is proved in \cite{BourbakiInt} that the paracompactness of $X/H$ guarantees the existence of a nonnegative bounded continuous function $\psi$ on $X$ such that $\supp{\psi}\cap K.H$ is compact whenever $K$ is a compact subset of $X$, and $\int_H\psi(xh)\,dh=1$ holds for all $x\in X$. Let $(\rho,V)$ be a unitary representation of $H$, the scalar product on $V$ being denoted by $\left(\cdot,\cdot\right)$. The corresponding representation of $C^*(H)$ will still be denoted $\rho$. For $\xi,\eta\in V$,

\begin{align*}
\left(\rho(\scal{f}{g})\xi,\eta\right) & = \int_H\gamma_X(h)\int_X\overline{f(x)}g(xh)\left(\rho(h)\xi,\eta\right)\,d\mu(x)\,dh\\
 & = \int_H\gamma_X(h)\int_X\overline{f(x)}g(xh)\left(\rho(h)\xi,\eta\right)\int_H\psi(xh')\,dh'\,d\mu(x)\,dh\\
 & = \int_X\psi(x)\int_{H\times H}\overline{f(xh'^{-1})}g(xh'^{-1}h)\left(\rho(h)\xi,\eta\right)\frac{\gamma_X(h)}{\delta_X(h')}\,dh'\,dh\,d\mu(x).
\end{align*}

For $u\in C_c(H)$ and $x\in X$, let $\tilde{u}_x(h)=\gamma_X(h)u(x.h)$ for all $h\in H$. Then,

\begin{align*}
\left(\rho(\scal{f}{g})\xi,\eta\right) & = \int_X\psi(x)\int_H\left(\tilde{f}_x^**\tilde{g}_x\right)(h)\left(\rho(h)\xi,\eta\right)\,dh\,d\mu(x)\\
& = \int_X\psi(x)\left(\rho(\tilde{f}_x^**\tilde{g}_x)\xi,\eta\right)\,d\mu(x)\label{eqpos}\tag{\dag}.
\end{align*}

The above computations are justified by the fact that the support of \[x\mapsto\int_H\left(\tilde{f}_x^**\tilde{g}_x\right)(h)\left(\rho(h)\xi,\eta\right)\,dh\] is contained in $\supp{f}.H\cap\supp{g}.H$ which has compact intersection with $\supp{\psi}$. Setting $f=g$ in (\ref{eqpos}), it is clear that $\scal{f}{f}$ is a positive element of $C_c(H)\subset C^*(H)$ and, taking $\rho$ to be faithfull, that $\scal{f}{f}=0$ implies $f=0$.
\end{proof}

\begin{definition}
Let $X$, $\mu$ and $H$ be as above. We denote by $\mathcal{E}_H(X,\mu)$ the Hilbert module over $C^*(H)$ obtained by completing $C_c(X)$ with respect to the norm $\|f\|=\left|\scal{f}{f}\right|_{C^*(H) }^{\frac{1}{2}}$ and extending the action of $C_c(H)$ to $C^*(H)$.
\end{definition}

\begin{notation}
When no confusion is likely to arise, this module will simply be denoted by $\EH$ or $\EX$.
\end{notation}

Let us now describe the left action of $C^*(G)$.

\begin{proposition}\label{Prop_leftaction}
Let $X$, $\mu$, $H$ and $G$ be as above. The action of $G$ on $X$ induces a $*$-morphism \[C^*(G)\longrightarrow\mathcal{L}_{C^*(H)}\left(\EX\right)\]
\end{proposition}

\begin{proof}
The action is first given at the level of the dense subalgebra $C_c(G)$. For $f\in C_c(X)$ and $\phi\in C_c(G)$, let \[(\phi.f)(x)=\int_G\phi(g)f(g^{-1}x)\,dg\] for all $x\in X$. Thus defined, $\phi.f$ belongs to $C_c(X)$ and Proposition \ref{scal} ensures that $\scal{\phi.f}{\phi.f}$ is a positive element of $C^*(H)$. Let $p$ be a state on $C^*(H)$, and $\nu_p$ the associated positive type Radon measure on $H$. Then the same computation as in the proof of Proposition \ref{scal} shows that \[p\left(\scal{f_1}{f_2}\right)=\int_H\scal{f_1}{f_2}(h)\,d\nu_p(h)=\int_X\psi(x)\int_H\left(\tilde{f_1}_x^**\tilde{f_2}_x\right)(h)\,d\nu_p(h)\,d\mu(x),\]  for $f_1,f_2\in C_c(X)$, using the same notations as above. It follows that the map $(f_1,f_2)\mapsto p\left(\scal{f_1}{f_2}\right)$ provides an inner product on $C_c(X)$. Consider the Hilbert space obtained from $C_c(X)$ by completion with respect to the associated norm, denoted $\left\|\cdot\right\|_{(p)}$. Left translations yield a representation $\pi_p$ of $G$ on this space and the $G$-invariance of $\mu$ implies that $\pi_p$ is unitary. Moreover, if $g\to g_0$ in $G$, then $\pi_p(g)f$ uniformly converges to $\pi_p(g_0)f$, while the supports remain in a fixed compact subset of $X$, so that $\pi_p$ is strongly continuous. Still noting $\pi_p$ for the integrated form of this representation, one has
\[p\left(\scal{\phi.f}{\phi.f}\right)=\left\|\pi_p(\phi.f)\right\|_{(p)}^2\leq\left\|\phi\right\|_{C^*(G)}^2\left\|\pi_p(f)\right\|_{(p)}^2=\left\|\phi\right\|_{C^*(G)}^2p\left(\scal{f}{f}\right).\] Since this inequality holds for any state $p$ of $C^*(H)$, it follows that \[\scal{\phi.f}{\phi.f}\leq\left\|\phi\right\|_{C^*(G)}^2\scal{f}{f}\] in $C^*(H)$. Straightforward computations finally show that the left action of $C^*(G)$ commutes to the right action of $C^*(H)$, and $\scal{\phi.f_1}{f_2}=\scal{f_1}{\phi^*.f_2}$ so that $C^*(G)$ acts by adjointable operators and the map $C^*(G)\rightarrow\mathcal{L}_{C^*(H)}\left(\EX\right)$ is a morphism of $C^*$-algebras.
\end{proof}

\begin{example}
If $X=G$, then by construction, the module $\EG$ is  the induction module $E_H^G$ introduced in \cite{RieffelIRC*A} by Rieffel.
\end{example}

\begin{example}
If $H=\left\{1\right\}$, then $\EX\simeq L^2(X,\mu)$, with the regular representation of $C^*(G)$ coming from the action of $G$ on $X$.
\end{example}

\begin{example}
If $X$ is reduced to a point, then $\EX$ may be identified to $C^*(H)$ considered as a Hilbert module over itself, $C^*(G)$ acting trivially.
\end{example}

The last two examples are in fact extreme cases of the following result, which describes $\EX$ when $X$ comes as the product of some topological space with the group acting on the right. This will also be the case in Section \ref{open}, when dealing with quotients of the open cell of some Bruhat decomposition.

\begin{theorem}\label{BxH}
Let $B$ be a paracompact Hausdorff space with a Borel measure $db$ and $H$ a locally compact group. Consider $X=B\times H$ with the action of $H$ given by right translations on itself and equipped with a measure of the form $d\mu(b,h)=\eta(h)\,db\,dh$ where $\eta$ is a continuous morphism from $H$ to $\R_+^*$. Then \[\EX\simeq L^2(B)\otimes C^*(H).\]
\end{theorem}

\begin{proof}
First notice that the particular form of the action of $H$ on $X$ implies that $\delta_X=\eta\Delta_H$, whence $\gamma_{X}=\sqrt{\eta}$. For $(f,g)\in C_c(B)\times C_c(H)$ and $(b,h)\in X$, let $P(f,g)(b,h)=\eta(h)^{-\frac{1}{2}}f(b)g(h)$. Then $P$ factorises through the algebraic tensor product $C_c(B)\otimes C_c(H)$ and the existence of a continuous partition of unity for $B$ implies that the range of $P$ is uniformly dense in $C_c(X)$. Simple caculations show that $P(f\otimes g*\varphi)=P(f\otimes g).\varphi$ for any $\varphi\in C_c(H)$, implying $C_c(H)$-linearity for $P$. One also has \[\scal{P(f\otimes g)}{P(f\otimes g)}=\left\|f\right\|_2^2\,g^*g\] in $C^*(H)$, so that $P$ extends to an isometry between the Hilbert modules $L^2(B)\otimes C^*(H)$ and $\EX$ with dense range, which ends the proof.
\end{proof}

\begin{remark}\label{rkBxH}
The previous proposition extends to the case where $B\times H$ is a subset of $X$ such that $\mu(X\setminus B\times H)=0$. In such a case, although the action of $G$ on $X$ might not restrict to an action on $B\times H$, the isometry between $\EX$ and $L^2(B)\otimes C^*(H)$ allows to define an action of $C^*(G)$ on the latter.
\end{remark}

\section{\texorpdfstring{The Hilbert module $\E$}{The Hilbert module E(G/N)}}\label{E(G/N)}

\subsection{Setting and notations}

In what follows, $G$ is a connected semi-simple Lie group with finite center. Let $K$ be a maximal compact subgroup in $G$, with Lie algebra $\k$. The Lie algebra $\g$ of $G$ admits Cartan decomposition $\g=\k+\p$, which determines a Cartan involution $\theta$. Let $\g=\k+\ap+\np$ be an Iwasawa decomposition of $\g$ and $G=K\Ap\Np$ the corresponding Iwasawa decomposition of $G$. The centraliser and the normaliser of $\Ap$ in $K$ are respectively denoted by $\Mp$ and $\M'p$. The group $B=\Mp\Ap\Np$ is called \textit{standard minimal parabolic subgroup} of $G$ and any closed subgroup $P$ of $G$ which is the normaliser of its Lie algebra and contains a conjugate of $B$ is called a \textit{parabolic subgroup}. Let $N$ be the \textit{unipotent radical} of such a subgroup, that is a maximal connected normal subgroup of $P$ consisting of unipotent elements. Let $L=P\cap\theta(P)$ be the \textit{$\theta$-stable Levi component} of $P$. It is a closed reductive subgroup such that $LN$ maps diffeomorphically to $P$. Let $A$ be a maximal connected split abelian subgroup in the center of $L$, and $M=\bigcap_{\chi}\Ker\left|\chi\right|$, where $\chi$ runs over the continuous homomorphisms $\chi:L\rightarrow\R^*$. Then $L=MA$. The decomposition $P=(M\times A)\ltimes N$ is called the \textit{Langlands decomposition} of $P$. Respectively denoting $\a$ and $\nk$ the lie algebras of $A$ and $N$, it is possible to choose an ordering on the $\a$-roots such that $\nk$ is the sum of the root-spaces associated to positive roots. Denoting $\g_\lambda$ the root-space associated to a root $\lambda$, let $\rho$ be the half sum of positive roots counted with multiplicity: $\rho=\frac{1}{2}\sum_{\lambda>0}\dim(\g_\lambda)\lambda$. Since $N$ is nilpotent and $L$ is reductive, they are both unimodular and the Haar measure $dp$ on $P=L\ltimes N$ decomposes into $dp=dl\,dn$. Finally, let $\nb$ denote the image $\theta\nk$ of $\nk$ under the Cartan involution and $\N$ be the corresponding analytic subgroup.

\subsection{The module \E}

The central object in this paper is the Hilbert module obtained by performing the construction of the previous section in the case of Lie groups $G$ and $L$ satifying the assumptions in the above setting, acting on the coset space $X=G/N$. As a consequence of Iwasawa decomposition, it is possible to write $G=KMAN$. Notice that $M\cap MAN = K\cap M$ is compact, and that the parts of the decomposition of an element in $G$ relatively to $KM$, $A$ and $N$ are unique. It follows from smoothness of the Iwasawa decomposition that the right action of $L$ on $G/N$ is free and proper. The coset space $(G/N)/L$ identifies to the flag variety $G/P$, hence is a compact manifold. Since $N$ is nilpotent and $G$ semisimple, these groups are both unimodular. It implies the existence of a $G$-invariant measure $\mu$ on $G/N$, unique up to normalisation. Exploiting the choice of a maximal compact subgroup $K$ of $G$ and once again the decomposition $G=KP$ and the fact that $N$ is normal in $P$, we may identify $G/N$ to $K\times M\times A$ as a topological space, and $\mu$ to $e^{2\rho\log (a)}dk\,dm\,da$.

The last identification immediately implies $L$-relative invariance for $\mu$ and allows to compute $\delta_{G/N}$. We recover these facts below, without using any choice of a maximal compact subgroup. In order to do so, let us first recall some classical notation.

\begin{notation}
Following \cite{BourbakiInt}, if $\Gamma$ is a topological group and $\sigma\in\text{Aut}(\Gamma)$, the \textit{modular function} of $\sigma$, denoted $\mathop{\rm{mod}}^\Gamma(\sigma)$, or $\mathop{\rm{mod}}(\sigma)$ when no confusion may result, is defined by the equality, $\int_\Gamma f\circ\sigma =\mathop{\rm{mod}}(\sigma)^{-1}\int_\Gamma f$, holding for any integrable function $f$. If a group $\Gamma'$ admits a subgroup $\Gamma$ which is normalised by an element $g\in\Gamma'$, we write $c_g(\gamma)=g\gamma g^{-1}$ for every $\gamma\in\Gamma$. In the case of an inner automorphism $c_\gamma$ of $\Gamma$, it follows from the definitions that $\mathop{\rm{mod}}^{\Gamma}(c_\gamma)=\Delta_\Gamma(\gamma^{-1})$.
\end{notation}

\begin{proposition}
The unique $G$-invariant measure $\mu$ on $G/N$ is relatively invariant with respect to the right action of $L$ and $\delta_{G/N}(l)=e^{2\rho\log(a)}$ for $l=ma\in L$.
\end{proposition}
\begin{proof}
We first prove that $\delta_{G/N}(l)=\mathop{\rm{mod}}^N(c_l)$ for any $l=ma\in L$. For $f$ integrable function and $\dot{g}$ the class in $G/N$ of $g\in G$, define $F(\dot{g}) = \int_N f(gn)\,dn$ so that our normalisations of Haar measures give $\int_G f = \int_{G/N}F\,d\mu$. Thus,
\begin{eqnarray*}
\delta_{G/N}(l)\int_G f(g)\,dg &  = & \int_{G/N} F(\dot{g}l^{-1})\,d\mu(\dot{g}) = \int_{G/N}\int_N f(gl^{-1}n)\,dn\,d\mu(\dot{g})\\
& = &\mathop{\rm{mod}}(c_l)\int_{G/N}\int_N f(g n l^{-1})\,dn\,d\mu(\dot{g})\\
& = &\mathop{\rm{mod}}(c_l)\int_G f(gl)\,dg = \mathop{\rm{mod}}(c_l)\int_G f(g)\,dg,
\end{eqnarray*}
hence the first part of the result and the expected equality. The rest of the proof reduces to the classical Lie algebra computation leading to the modular functions of parabolic subgroups. Identifying $N$ to its Lie algebra \textit{via} the exponential map, one only needs to compute the jacobian determinant $\left|\det(\Ad(l)|_\nk)\right|$. The properties of Langlands decomposition imply that $M$ is the product of a closed subgroup of $K$, hence compact and a connected reductive group with compact center. It follows that $\left|\det(\Ad(m)|_\nk)\right|=1$ for any $m\in M$. Finally, an element $a\in A$ acts on the root space $\g_\lambda$ by $e^{\lambda\log (a)}$, so that $\left|\det(\Ad(a)|_\nk)\right|=e^{2\rho \log(a)}$, which concludes the proof.
\end{proof}

The data of \[G\curvearrowright(G/N,\mu)\curvearrowleft L\] hence satisfies the assumptions required in the previous section. Let us now turn to the properties of $\E$, starting with the fact that it induces the $P$-series representations.

\subsection{Specialisation}\label{specialisation}

As explained in Section \ref{intro}, Rieffel's modules allow to recover the particular induced representations by tensoring with the Hilbert space of the inducing representation. This specialisation procedure is still available using the $C^*(L)$-module $\E$. The existence of the corresponding maps essentially relies on the following result, relating $\E$ to Rieffel's module $E_P^G=\mathcal{E}_P(G)$.

\begin{proposition}\label{Prop_EG_EG/N}
There is an isometric isomorphism of $C^*(L)$-Hilbert modules \[\EL\simeq\EP\otimes C^*(L),\] that intertwines the left actions of $C^*(G)$ on both sides.
\end{proposition}

\begin{proof}
Let $M_N$ denote the averaging map defined by $M_N(f)(\dot{g})=\int_N f(gn)\,dn$ for any function $f$ in $C_c(G)$ and $\dot{g}\in G/N$ the class of $g\in G$. A similarly defined map on $C_c(P)$ extends to a surjection $\varepsilon_N:C^*(P)\twoheadrightarrow C^*(L)$. Let $\alpha\in C_c(P)\subset C^*(P)$, and $f\in C_c(G)\subset\EG$. Using the decomposition of measure $dp=dl\,dn$ it is easily seen that \[M_N(f.\alpha)=M_N(f).\varepsilon_N(\alpha)\label{MN}\tag{$*$}\] in $\E$. Now for $l=ma\in L$, \[\gamma_{G,P}(l)=\sqrt{\frac{\Delta_G}{\Delta_P}}(l)=\Delta_P^{-\frac{1}{2}}(l)=e^{\rho\log(a)}=\delta_{G/N}^{\frac{1}{2}}(l)=\sqrt{\frac{\delta_{G/N}}{\Delta_L}}(l)=\gamma_{G/N,L}(l).\]
Then it follows from a straightforward computation that \[\scal{M_N(f_1)}{M_N(f_2)}_{\E}=\varepsilon_N\left(\scal{f_1}{f_2}_{\EG}\right)\label{proj_scal}\tag{$**$}\] in $C^*(L)$ for every $f_1,f_2\in C_c(G)$. Recall that the $C^*(L)$-valued inner product on $\EG\otimes C^*(L)$ is defined on elementary tensors by \[\scal{f_1\otimes\varphi_1}{f_2\otimes\varphi_2}=\scal{\varphi_1}{\varepsilon_N\left(\scal{f_1}{f_2}_{\EG}\right)\varphi_2}_{C^*(L)}=\varphi_1^*\varepsilon_N\left(\scal{f_1}{f_2}_{\EG}\right)\varphi_2.\]
Let $A$ be defined on $C_c(G)\times C_c(L)$ by $A(f,\varphi)=M_N(f).\varphi$. Equality (\ref{MN}) proves that $A$ factorises through $C_c(G)\otimes C_c(L)$, while (\ref{proj_scal}) implies that \[\scal{A(f_1\otimes\varphi_1)}{A(f_2\otimes\varphi_2)}_{\E}=\scal{f_1\otimes\varphi_1}{f_2\otimes\varphi_2}_{\EG\otimes C^*(L)}.\]
Since $M_N$ is onto, it follows that $A$ also has dense range and the above equality shows that it extends to the expected isometric isomorphism.
\end{proof}

\begin{remark}
The modules $E_P^G\otimes C^*(L)$ were used in \cite{Pierrot} by F. Pierrot. The interest of the approach to $\E$ as a special case of the general construction of Section \ref{construction} is so far two-fold: first, it leads to the convenient realisation of $\E$ given in Theorem \ref{Thm_Open} through the use of the general result describing $\EX$ when $X=B\times H$ (Proposition \ref{BxH}). Another interest of seeing $\E$ as a completion of $C_c(G/N)$ is that it makes it possible to directly define intertwining integrals similar to the ones considered by A. W. Knapp and E. M. Stein in \cite{KS1, KS2} without needing any of the meromorphic continuation argument used by these authors, as will be seen in Section \ref{intertwine}. 
\end{remark}

Let us now establish the existence of specialisation maps.

\begin{corollary}[Specialisation]
For $(\sigma,\chi)\in\widehat{M}\times\widehat{A}$, let $\mathcal{H}_{\sigma\otimes\chi}$ be the Hilbert space of the representation 
$\sigma\otimes\chi\otimes\rm{1}$ of $P$, and $\mathcal{H}_{\sigma,\chi}^P$ the space of the induced representation \[\pi^P_{\sigma,\chi}=\ind{P}{G}{\sigma\otimes\chi\otimes\rm{1}}.\] There exists a map \[q_{\sigma,\chi}:\EL\otimes_{C^*(L)}\mathcal{H}_{\sigma\otimes\chi}\longrightarrow\mathcal{H}_{\sigma,\chi}^P\] unitarily intertwining the actions of $C^*(G)$ on these Hilbert spaces.
\end{corollary}

\begin{proof}
Associativity of the tensor product allows to reduce the proof to applying the similar result obtained by Rieffel in the case of classical induction Hilbert modules. Proposition \ref{Prop_EG_EG/N} provides a unitary equivalence \[\EL\otimes_{C^*(L)}\mathcal{H}_{\sigma\otimes\chi}\simeq\left(\EP\otimes_{C^*(P)}C^*(L)\right)\otimes\mathcal{H}_{\sigma\otimes\chi}.\] The result follows from composing this isomorphism with the specialisation maps of \cite[Theorem 5.12 p.228]{RieffelIRC*A}.
\end{proof}

\begin{remark}
Letting $P$ be a cuspidal parabolic subgroup and $\sigma$ run over $\Md$, it appears that, although $\E$ is a Hilbert module over $C^*(L)$, it still `contains' the $P$-series representations of $G$, in the sense that it induces all of them. 
\end{remark}

We now turn to other realisations of $\E$, encoding some classical features of $P$-series representations.

\section{Different pictures}\label{pictures}

Notations in this section are the same as in the previous one.

\subsection{Induced picture}\label{induced}

According to the usual definition, essentially due to G. Mackey for locally compact groups, induced representations act on spaces of sections of equivariant fiber bundles. In particular, for $(\sigma,\chi)$ in $\Md\times\widehat{A}$, the $P$-series representation $\pi^P_{\sigma,\chi}$ acts on the space of $L^2$-sections of the fiber product $G\times_{\sigma\otimes\chi\otimes\rm{1}}\mathcal{H}_{\sigma\otimes\chi\otimes\rm{1}}$ over the flag manifold $G/P$. The trivial behaviour of the inducing parameter on $N$ allows to consider sections of the $G$-equivariant bundle
\[\xymatrix{G/N\times_{\sigma\otimes\chi}\ar[d]\mathcal{H}_{\sigma\otimes\chi}\\
G/P}\] as the space of $\pi^P_{\sigma,\chi}$. In order to recover that point of view within the global approach provided by the module $\E$, it is tempting to try and realise it as a space of sections of the $G$-equivariant bundle
\[\xymatrix{G/N\times_L\ar[d]C^*(L)\\
G/P}\] where $C^*(L)$ can be viewed as the collection of all the Hilbert spaces $\mathcal{H}_{\sigma\otimes\chi\otimes\rm{1}}$. This realisation will be called the \textit{induced picture} of $\E$.

Let us denote by $\Eio$ the space of compactly supported continuous functions $F:G/N\rightarrow C^*(L)$ subject to the relation \[F(x.l)=e^{\rho\log(a)}U_{l^{-1}}.F(x)\] for any $x\in G/N$ and $l=ma\in L$, on which $C^*(L)$ acts by right multiplication. Let $\psi$ be a Bruhat section for the action of $L$ on $G/N$, as defined in the proof of Proposition \ref{scal}. The existence of $\psi$ follows from the (para)compactness of the coset space $(G/N)/L\simeq G/P$. A $C^*(L)$-valued inner product is defined on $\Eio$ by considering \[\scal{F_1}{F_2}_\psi=\int_{G/N}F_1(x)^*F_2(x)\psi(x)\,d\mu(x)\] for $F_1,F_2\in\Eio$. The problem of the dependence on $\psi$ is settled by the following result.

\begin{lemma}
The map $\scal{\cdot}{\cdot}_\psi$ defined above does not depend on the choice of the Bruhat section $\psi$.
\end{lemma}

\begin{proof}
Let $\psi_1$ and $\psi_2$ be two Bruhat sections on $G/N$ and $u=\psi_1-\psi_2$. For $x\in G/N$ be represented by $kma\in KMA$, the relation satisfied by any functions $F_1,F_2\in\Eio$ implies that $F_1(x)^*F_2(x)=e^{-2\rho\log(a)}F_1(k)^*F_2(k)$. It follows from the Iwasawa decomposition of the measure that
\begin{eqnarray*}
\int_{G/N}F_1(x)^*F_2(x)u(x)\,d\mu(x)&=&\int_{K\times MA}F_1(kma)^*F_2(kma)e^{2\rho\log(a)}\,dk\,dm\,da\\
&=&\int_K F_1(k)^*F_2(k)\,dk\int_{L}u(kl)\,dl\,dk
\end{eqnarray*}
The last quantity vanishes since $\int_{L}u(kl)\,dl=0$ for all $k\in K$ by definition of $\psi_1$ and $\psi_2$. It follows that $\scal{F_1}{F_2}_{\psi_1}=\scal{F_1}{F_2}_{\psi_2}$.
\end{proof}

\begin{notation}
As a consequence of the above lemma, we can denote without ambiguity $\scal{\cdot}{\cdot}_i$ the sesquilinear form on $\Eio$, regardless of the Bruhat section used to construct it.
\end{notation}

\begin{definition}[Induced picture]
The space obtained by completing $\Eio$ with respect to the norm induced by $\left|\scal{\cdot}{\cdot}_i\right|$ and denoted $\Ei$ is called the \textit{induced picture of $\E$}.
\end{definition}

The $C^*(L)$-module $\Ei$ carries a left action of $C^*(G)$, defined by convolution at the level of compactly supported functions, in the same way as the one of Proposition \ref{Prop_leftaction}.

\begin{proposition}
The map $f\mapsto\tilde{f}$ defined on the dense subset $C_c(G/N)$ of $\E$ by \[\tilde{f}(x)(l)=e^{\rho\log(a)}f(x.l)\] for $x\in G/N$ and $l=ma\in L$, takes values in $\Eio$. It is $C_c(L)$-linear and preserves the $C^*(L)$-valued inner products.
\end{proposition}

\begin{proof}
Let $f\in C_c(G/N)$, $x\in G/N$ and $l_0=m_0a_0,l=ma\in L$. Then $\tilde{f}(x)\in C_c(L)$ and $\tilde{f}(xl)(l_0)=e^{\rho\log (a_0)}f(xll_0)$. Since \[\left[U_{l^{-1}}\tilde{f}(x)\right](l_0)=\tilde{f}(x)(ll_0)=e^{\rho\log (aa_0)}f(xll_0),\] the expected relation $\tilde{f}(xl)=e^{-\rho\log (a)}U_{l^{-1}}\tilde{f}(x)$ holds, implying that $\tilde{f}\in\Eio$. The $C_c(L)$-linearity follows from a straightforward computation. We prove the isometry property: for $f_1,f_2\in C_c(G/N)$,
\begin{eqnarray*}
\scal{\widetilde{f_1}}{\widetilde{f_2}}_i(l_0)&=&e^{\rho\log(a_0)}\int_{G/N}\int_L\overline{\widetilde{f_1}(x)(l)}\widetilde{f_2}(x)(ll_0)\,dl\,\psi(x)\,d\mu(x)\\
&=&e^{\rho\log(a_0)}\int_{G/N}\int_L e^{2\rho\log (a)}\overline{f_1(xl)}f_2(xll_0)\psi(x)\,dl\,d\mu(x)\\
&=&e^{\rho\log (a_0)}\int_{G/N}\overline{f_1(x)}f_2(xl_0)\int_L \psi(xl^{-1})\,dl\,d\mu(x)\\
&=&\scal{f_1}{f_2}_{\E}(l_0)\int_L \psi(xl)\,dl = \scal{f_1}{f_2}_{\E}(l_0)
\end{eqnarray*}
where $\psi$ is any Bruhat section on $G/N$.
\end{proof}

The map of the above statement has dense range in $\Ei$, as we see by considering $F\in\Eio$ such that $F(x)\in C_c(L)$ for $x\in G/N$: then $F=\tilde{u}$ where $u:x\mapsto F(x)(1)$. The properties of $C^*(L)$-sesquilinearity and positivity of $\scal{\cdot}{\cdot}_i$ can be obtained as consequences of the previous proposition, and the following theorem holds as an immediate corollary.

\begin{theorem}
There is an isometric isomorphism of $C^*(L)$-Hilbert modules \[\E\simeq\Ei.\] Moreover, the left action of $C^*(G)$ on $\Ei$ given by convolution coincides with the one obtained by transporting it from $\E$ \textit{via} this isomorphism.
\end{theorem}

\begin{remark}
It is clear from the definition and Iwasawa decomposition that functions in $\Eio$ are determined by their restriction to $K$. It makes it possible to obtain a \textit{compact picture} of $\E$ as the completion of a space of functions $K\rightarrow C^*(L)$.
\end{remark}

\subsection{Open picture}\label{open}

The classical so-called \textit{open} or \textit{noncompact} picture of $P$-series representations (see \cite{Knapp1}) allows to realise all these representations on Hilbert spaces which do not depend on the representation $\sigma$ of $M$ in the inducing parameter. The crucial observation is the following consequence of the Bruhat decomposition of $G$ (see for instance \cite{KnappBeyond}):

\begin{proposition}[Open Bruhat cell]\label{OBC}
The set $\N MAN$ is open in $G$ and its complement has Haar measure $0$.
\end{proposition}

The above fact is reflected by an isomorphism between $\E$ and the tensor product of a Hilbert space by the right-acting $C^*$-algebra:

\begin{theorem}[Open picture]\label{Thm_Open}
There is an isometric isomorphism of $C^*(L)$-Hilbert modules \[\E\simeq L^2(\N)\otimes C^*(L).\]
\end{theorem}

\begin{proof} It is a straightforward consequence of Theorem \ref{BxH} and Remark \ref{rkBxH}, applied to $B=\N$ and $H=L$, for $\N L$ has measure $0$ in $G/N$. More precisely, a dense submodule of $\E$ is obtained by considering functions of the form \[F:\bn ma\longmapsto e^{-\rho\log(a)}f(\bn)\varphi(ma),\] where $f\in C_c(\N)$ and $\varphi\in C_c(L)$.
\end{proof}

In the two following sections, we turn to applications of the point of view provided by the module $\E$ on $P$-series. Section \ref{irred} is devoted to the characterisation of bounded self-intertwiners of $\E$, interpreted as an irreducibility result. In Section \ref{intertwine}, we define intertwining integrals in the spirit of \cite{KS1} on a dense subset of $\E$.

\section{An irreducibility theorem}\label{irred}

\subsection{\texorpdfstring{Groups of real rank $1$}{Groups of real rank 1}}\label{RR1}

In this section, we will sometimes need to assume that the real rank of $G$, that is the dimension of the abelian part in the Iwasawa decomposition of $G$, is $1$. As a consequence, proper parabolic subgroups of $G$ are necessary minimal. The subgroups in the Langlands decomposition of $P=B$ are $M=M_\p$ compact, $A=A_\p\simeq\R_+^*$ and $N=N_\p$.

\begin{remark}
The $P$-series induced from a minimal parabolic subgroup are also called \textit{principal series}.
\end{remark}

We denote:
\begin{itemize}
\item $\alpha$ the smallest restricted root of $(\g:\a)$
\item $p$ and $q$ the respective dimensions of $\g_{-\alpha}$ and $\g_{-2\alpha}$ 
\end{itemize}

It follows that $\nb=\g_{-\alpha} + \g_{-2\alpha}$ and $\rho=\frac{1}{2}(p+2q)\alpha$. The automorphisms $\left\{c_a\,,\,a\in A\right\}$ of $\N$ are called \textit{dilations}.

The Weyl group $W=N_K(A)/Z_K(A)$ contains exactly one non-trivial element, denoted $w$, and the Bruhat decomposition of $G$ writes $G=PwP\sqcup P$, so that the complement of $\N MAN$ in $G$ is the class $wP$ of $w$ in $G/P$.

\begin{notation}\label{nman}
Elements in the dense subset $\left(G/N\right)\setminus wL$ may be written according to $\N MA$ in a unique way. For such an element this decomposition will be denoted \[g=\n(g)\m(g)\ab(g)n.\] We also denote $\lm(g)=\m(g)\ab(g)$.
\end{notation}

What follows essentially consists in analysing distributions on $\N$ with homogeneity properties under the action of the one-parameter group of dilations $A$. Identifying $\N$ to its Lie algebra $\nb$, this action is seen to dilate vectors with different coefficients according to root subspaces. This behaviour is taken into account by a special norm-like function on $\N$.

\begin{definition}
The \textit{norm function on $\N$} is the function defined on $\N\setminus\left\{1\right\}$ by \[|\bar{n}|=e^{\rho\log \ab(w^{-1}\bar{n})}.\]
\end{definition}

Elementary facts about this function may be found in \cite{KS1} and \cite{These}. In particular, it will be useful to know that the norm function is $C^\infty$ on $\N\setminus\left\{1\right\}$ and is continuously extended to $\N$ by setting $|1|=0$. Moreover, $|\bn^{-1}|=|\bn|$ for any $\bn\in\N$ and the measure $\frac{d\bn}{|\bn|}$ is invariant under dilations.

Another feature in real rank $1$, is the possibility to completely describe the left action of $C^*(G)$ in the open picture of $\E$. It is initially defined by convolution at the level of $C_c(G/N)$ and carried to $L^2(\N)\otimes C^*(L)$ \textit{via} the isomorphism of Theorem \ref{Thm_Open}. One may also consider the corresponding action of $G$ defined by translations on $C_c(G/N)$ and extended to $\E$. The next proposition explicits this action in the open picture. For $f\otimes\varphi$ in $L^2(\N)\otimes C^*(L)$, we write $g.f\otimes\varphi$ for the action of $g\in G$ transported from the one on $\E$. The Bruhat decomposition in rank $1$ proves that it is enough to consider $g\in\N MA$ and $g=w$.

\begin{proposition}\label{leftactionopenpict}
Let $f\otimes\varphi\in L^2(\N)\otimes C^*(L)$, $\bn_0\in\N$ and $l_0=m_0a_0\in MA$. Then,
\begin{itemize}
\item $\bn_0.(f\otimes\varphi) = \lambda_{\N}(\bn_0)(f)\otimes\varphi$
\item $l_0.(f\otimes\varphi) = e^{\rho\log(a_0)}f\circ c_{l_0^{-1}}\otimes U_{l_0}.\varphi$
\item For any $\bnu\in\N\setminus\left\{1\right\}$ and $\lambda\in L$, \[w.(f\otimes\varphi)(\bnu,\lambda)=\frac{1}{|\bnu|}f(\n(w^{-1}\bnu))\,\left(U_{\lm(w^{-1}\bnu)}\varphi\right)(\lambda).\]
\end{itemize}
\end{proposition}

\begin{proof}
It follows directly from the definition of the isomorphism in Proposition \ref{Thm_Open}.
\end{proof}

Let us now turn to the main result of this section.

\subsection{Bounded self-intertwiners}

\begin{definition}
Let $A$ be a $C^*$-algebra. An element $M$ in the multiplier algebra $\mathcal{M}(A)=\mathcal{L}_A(A)$ is said to be \textit{central} if it satisfies the relation \[M(ab)=aM(b)\] for any $a,b\in A$.
\end{definition}

\begin{remark}
Using approximate units, the algebra of central multipliers of $A$ can be proved to coincide with the center of $\mathcal{M}(A)$.
\end{remark}

\begin{theorem}\label{Thm_Irred}
The elements of $\mathcal{L}_{C^*(L)}(\E)$ which commute to the left action of $C^*(G)$ are exactly the central multipliers of $C^*(L)$.
\end{theorem}

The method of the proof consists in associating to such an operator a bilinear form on a submodule of test functions and study the properties of the associate distributional kernel.

In what follows, all distributions take values in Banach spaces. General theory may be found in \cite{Bruhat}, as well as the next proposition, which characterises the distributions satisfying some invariance properties.

\begin{proposition}\label{distribinvar}
Let $M$ be a differentiable manifold and $\Gamma$ a Lie group with Haar measure $d\gamma$. Let $E$ be a Banach space and $T$ an $E$-valued distribution on $M\times\Gamma$. If $T$ is invariant under the transformations $(m,\gamma)\mapsto(m,\gamma_0\gamma)$, then there exists an $E$-valued distribution $S$ on $M$ such that \[\scal{T}{\varphi} = \int_{\Gamma}\scal{S}{\varphi_\gamma}\,d\gamma\] where $\varphi_\gamma:m\mapsto\varphi(m,\gamma)$ whenever $\varphi$ is a test function on $M\times\Gamma$.
\end{proposition}


\begin{remark}\label{remdistribinvar}
As a special case, it follows that, up to a scalar factor, the only left-invariant distribution on a Lie group is the Haar measure.
\end{remark}

\begin{proof}[Proof of Theorem \ref{Thm_Irred}]

The right action of $C^*(L)$ extends to one of $\mathcal{M}(C^*(L))$ on $\E$ and if $T_M$ denotes right multiplication by $M\in\mathcal{M}(C^*(L))$, the centrality condition for $M$ implies $C^*(L)$-linearity for $T_M$. The fact that $T_M$ commutes to the action of $C^*(G)$ is trivial, and boundedness follows from the existence of an adjoint map for $T_M$, namely $T_{M^*}$.

Conversely, let $T\in\mathcal{L}_{C^*(L)}(\E)$, satisfying the centrality condition. Following notations introduced in Section \ref{intro}, $\mathcal{M}(\E)$ denotes the $\mathcal{M}(C^*(L))$-Hilbert module $\mathcal{L}(C^*(L),E)$. 

We shall work in the open picture of Section \ref{open}. Since $\mathcal{M}(\E)$ contains $L^2(\N)\otimes\mathcal{M}(C^*(L))$, there is an injection $L^2(\N)\hookrightarrow\mathcal{M}(\E)$ through which $f\in L^2(\N)$ is identified with the multiplier $f\otimes 1_{\mathcal{M}(C^*(L))}$, also denoted $m_f$, so that $m_{f\otimes a}=m_f m_a$, for any $a\in C^*(L)$, using notations of Section \ref{generalnotations}. It also follows that \[M(T)(m_{f\otimes a}) = M(T)(m_f)m_a,\] hence for $f_1\otimes a_1,f_2\otimes a_2\in\E$, \[\scal{M(T)(m_{f_1\otimes a_1})}{m_{f_2\otimes a_2}} = m_{a_1^*}\scal{M(T)(m_{f_1})}{m_{f_2}}m_{a_2}.\]

Let us now consider the map $B_T:C_c(\N)\times C_c(\N)\rightarrow\mathcal{M}(C^*(L))$ defined by
\[B_T(f_1,f_2) = \scal{M(T)(m_{\overline{f_1}})}{m_{f_2}},\]
and prove that it is a Radon measure on $\N\times\N$. Let $K$ be a compact subset in $\N\times\N$, and $f_1,f_2\in C_c(\N)$ such that $\supp{f_1}\times\supp{f_2}\subset K$. Recall that if $E$ is a Hilbert module over a $C^*$-algebra $A$, the identity
$\left|\scal{\xi}{\eta}\right|_A\leq \|\xi\|\scal{\eta}{\eta}^{\frac{1}{2}}$ holds for any $\xi,\eta\in E$, hence the following equality in $\mathcal{M}(C^*(L))$:
\[\left|B_T(f_1,f_2)\right|\leq\|M(T)(m_{f_1})\|.\left|m_{f_2}\right|.\]
It follows that $\|B_T(f_1,f_2)\|\leq \|T\|.\|f_1\|_2.\|f_2\|_2$, where $\|T\|$ denotes the operator norm of $T$, hence continuity of $B_T$ with respect to the topology of uniform convergence on $K$. Consequently, $B_T$ defines a distributional kernel $k_T$ on $\N\times\N$, so that it writes
\[B_T(f_1,f_2)=\int_{\N\times\N}f_1(\bn_1)f_2(\bn_2)\,k_T(\bn_1,\bn_2)\,d \bn_1\,d\bn_2.\]
Since the left action of $G$ on $\E$ preserves the inner product, commutation of $T$ to this action implies that $B_T(\bn_0.f_1,\bn_0.f_2)=B_T(f_1,f_2)$.

Applying the diffeomorphism $(\bn_1,\bn_2)\mapsto(\bn_1^{-1}\bn_2,\bn_2)$ of $\N\times\N$ and using Proposition \ref{distribinvar} and Remark \ref{remdistribinvar} we see that $k_T$ satisfies the equation \[k_T(\bn_1,\bn_2)=k_T(1,\bn_1^{-1}\bn_2).\]
It follows that defining a distribution $c_T$ on $\N$ by $c_T(\bn) = k_T(1,\bn)$ implies that $k_T(\bn_1,\bn_2) = c_T(\bn_1^{-1}\bn_2)$. Invariance under the $A$ action will allow us to characterise this distribution. Namely, the action of an element $l=ma\in L$ on an elementary tensor $f\otimes\varphi\in L^2(\N)\otimes C^*(L)$ is given by the formula $l.f\otimes\varphi =e^{\rho\log(a)} f\circ c_l\otimes U_l.\varphi$ of Proposition \ref{leftactionopenpict}.

Commutation to the $L$ action implies that $B_T(a.f_1,a.f_2) = B_T(f_1,f_2)$ for any $a\in A$. It follows that
\begin{eqnarray*}
e^{-2\rho\log(a)}&&\hspace{-8mm}\int_{\N\times\N}f_1(c_{a}(\bn_1))f_2(c_{a}(\bn_2))\,k_T(\bn_1,\bn_2)\,d\bn_1\,d\bn_2\\
&=&e^{2\rho\log(a)}\int_{\N\times\N}f_1(\bn_1)f_2(\bn_2)\,k_T(c_{a}^{-1}(\bn_1),c_{a}^{-1}(\bn_2))\,d\bn_1\,d\bn_2\\
&=&\int_{\N\times\N}f_1(\bn_1)f_2(\bn_2)\,k_T(\bn_1,\bn_2)\,d\bn_1\,d\bn_2,
\end{eqnarray*}
the first equality resulting of $\text{mod}^{\N}(c_a)=e^{-2\rho\log(a)}$. As a consequence the distribution $c_T$ satisfies the following invariance property: $c_T(c_a(\bn))= e^{-2\rho\log(a)}c_T(\bn)$, that is, for any test function $\varphi$ on $\N$, \[\hspace{3cm}\scal{c_T}{\varphi\circ c_a} = \scal{c_T}{\varphi}.\tag{$\ddag$}\label{cTinvar}\]

\begin{proposition}
If the real rank of $G$ is $1$, then the Radon measures on $\N$ satisfying relation (\ref{cTinvar}) are multiples of the Dirac measure.
\end{proposition}

\begin{proof}

Since $\N$ is a simply connected nilpotent Lie group, it may be identified \textit{via} the exponential map to its Lie algebra $\nb=\g_{-\alpha}\oplus\g_{-2\alpha}\simeq\R^p\oplus\R^q$, in a way that preserves measures and allows to identify spaces of test functions and distributions. Under these identifications, the action by dilations of $A$ on $\N$ is given on $\nb$ in terms of $\alpha$ by $c_a(\bn)\simeq(\alpha(a)u,\alpha(a)^2v)$ for $a\in A$ and $\bn\in\N$ identified to $(u,v)\in\R^p\oplus\R^q$. We denote by $a.(u,v)$ this last expression.

Denote $r(u,v)=\left(\|u\|^4 + \|v\|^2\right)^{\frac{1}{4}}$ for $(u,v)\in\R^p\oplus\R^q$ and let $c_T^0$ be the restriction of $c_T$ to the open subset $\N\setminus\left\{1\right\}$. The function $r$ is the Lie algebraic analogue of the norm function introduced above. For $t\in\R_+$, we denote the surface of equation $r(u,v)=t$. Then, for $a\in A\simeq\R_+^*$, we have $r(a.(u,v))=a\,r(u,v),$ and the map \[(u,v)\longmapsto(r(u,v),\frac{(u,v)}{r(u,v)})\] is a diffeomorphism between $\N\setminus\left\{1\right\}$ and $\R_+^*\times S_1$. Fix $\psi_0\in\test(S_1)$. If $\varphi$ is a test function on $\R_+^*$, then $\varphi\otimes\psi_0$ is in $\test(\R_+^*)\otimes\test(S_1)\subset\test(\N\setminus\left\{1\right\})$ and the map $\varphi\mapsto\scal{c_T^0}{\varphi\otimes\psi_0}$ is a homogeneous distribution on $\R_+^*$. Proposition \ref{distribinvar} implies that $\scal{c_T^0}{\varphi\otimes\psi_0}$ is of the form \[c(\psi_0)\int_0^{+\infty}\varphi(r)\frac{dr}{r}\] so $c_T^0$ cannot be the restriction of a Radon measure. It follows that the support of $c_T$ is reduced to $\left\{1\right\}$, so that $c_T$ is a combination of derivatives in the sense of distributions of the Dirac measure. The homogeneity condition finally proves that $c_T = \delta_1.1_{\mathcal{M}(C^*(L))}$ up to a constant.
\end{proof}

According to the above result, there exists a multiplier $U\in\mathcal{M}(C^*(L))$ such that $c_T= U\delta_1$. It follows that $k_T(\bn_1,\bn_2)=U.\delta_1(\bn_1^{-1}\bn_2)$, hence the following form of the bilinear form: \[B_T(f_1,f_2)=\scal{f_1}{f_2}_{L^2}\;U\]
for $f_1,f_2\in L^2(\N)$.

As a consequence, for $a_1,a_2\in C^*(L)$,
\begin{eqnarray*}
\scal{M(T)(m_{f_1}\otimes m_{a_1})}{m_{f_2}\otimes m_{a_2}} &=& \scal{f_1}{f_2}_{L^2}\;m_{a_1}^*\;U\;m_{a_2}\\
&=&\scal{f_1\otimes U^*(a_1)}{f_2\otimes a_2}
\end{eqnarray*}
in $\mathcal{M}(C^*(L))$, and $M(T)$ coincides with right composition by $U^*$. This map is $C^*(L)$-linear only if $U$ satisfies the centrality condition, which concludes the proof of Theorem \ref{Thm_Irred}.

\end{proof}
\begin{remark}
In the special case of $q=0$, the distribution $c_T$ is homogeneous of degree $-p$. If $G=\GR$ for instance, $p=1$ and the problem reduces to the Euler equation on $\R$, the solutions of which are known to be combinations of the Dirac measure $\delta_0$ and the \textit{principal value} distribution $\text{Vp}\left(\frac{1}{x}\right)$. The latter being of order $1$, it is not the restriction of a Radon measure, which implies that $c_T = \delta_0$ up to a constant factor.
\end{remark}

The above result should be seen as the analogue at the level of Hilbert modules of the generic irreducibility theorem due to Harish-Chandra in the general cuspidal case, and to Bruhat in the minimal one (see \cite{Lipsman}). More precisely, they established the irreducibility of all $P$-series representations except for the ones induced by elements in $\widehat{M}_d\times\widehat{A}$ fixed under the action of the Weyl group $W_P$, for which reducibility may happen. In view of the Schur Lemma, it means that the self-intertwiners for these representations amount to homotheties, except for the ones coming from a subset of lower dimension in $\widehat{L}$.

In our global framework, Theorem \ref{Thm_Irred} states that self-intertwiners reduce to what plays the role of homotheties on a module over a non-commutative algebra, that is multipliers satisfying the extra centrality condition needed to be linear. That all the bounded operators commuting to $C^*(G)$ are trivial in this sense, reflects the fact that the subset of parameters in $\widehat{L}$ giving reducibility is too small to manifest at the level of $\E$.

\section{Standard intertwining integrals}\label{intertwine}

The generic irreducibility theorem mentioned above states that, with respect to the Plancherel measure, almost all the $P$-series are irreducible. To deal with the representations in the remaining measure zero set, it is necessary to determine if non-trivial self-intertwiners may exist. Knapp and Stein successfully developed such a theory of intertwining operators in \cite{KS1,KS2}, some aspects of which are still being discovered. The first applications were related to the detection of fine reducibility phenomena in the $P$-series and explicit computation of densities in the Plancherel formula (see \cite{Knapp1}).

Their method starts from Bruhat's theory and the remark that certain integrals formally enjoy the intertwinig properties although they are given by non-locally integrable kernels. The constructions then roughly proceeds in two steps. The first one consists in allowing the parameter in $\widehat{A}$ seen as a subset of $\a'\otimes\C$ to take non-purely imaginary values. The corresponding integrals are then convergent on a domain and the operators may be defined as meromorphic extensions. The second step consists in normalising these operators, to make them unitary. The normalising functions are related to densities in the Plancherel formula and the study of their poles allows to decide of the reducibility of induced representations.

Let us now turn to standard intertwining integrals. As explained at the beginning of the previous paragraph, the theory of intertwining operators relies on the possibility to give meaning to certain integral formula. In the classical context, the integral operators are meant to intertwine the representations $\pi_{\sigma,\chi}$ and $\pi_{w.\sigma,w.\chi}$. One is led to study the so-called \textit{standard operator} $I_w$ given by \[I_w F(g)=\int_{\N} F(gw\bn)\,d\bn,\] which formally turns a section $F$ on $G/LN$ into a section on $G/L\N$.

At the level of the generalised induction modules introduced above, intertwining operators should be $C^*(G)$-invariant operators $\mathcal{E}(G/N)\longrightarrow\mathcal{E}(G/\N)$, preferably preserving the inner products. As a first step in the direction of such a theory, we study what can be defined by considering the integral formula $I_w$ on a dense submodule of $\E$. We will prove the following statement:

\begin{theorem}\label{convergence}
The standard integral \[\int_{\N} F(gw\bn)\,d\bn\] defines a linear map $C_c(G/N)\rightarrow C(G/N)$.
\end{theorem}

\subsection{Proof of the convergence}

Recall that $G$ diffeomorphically decomposes into $KAN$. Hence any $g\in G$ may be written accordingly $g=\underline{k}(g)\underline{a}(g)\underline{n}(g)$. The proof of Theorem \ref{convergence} essentially relies on the following property of the map $\underline{a}$.

\begin{lemma}\label{properness}
The restriction of the Iwasawa projection \[\underline{a}|_{\N}:\N\longrightarrow A\] is proper. 
\end{lemma}

\begin{proof}
Let us first assume that $G$ has real rank one. The result can be obtained in this case using so-called $\mathrm{SU}(2,1)$-reduction as explained in \cite{HelgasonDS}. More precisely, let us assume that $\alpha$ and $2\alpha$ are the positive restricted roots, so that $\nb=\g_{-\alpha}+\g_{-2\alpha}$, to which $\N$ identifies. Then, writing $\bn=\exp (X+Y)$ with $X\in\g_{-\alpha}$ and $Y\in\g_{-2\alpha}$, Theorem 3.8 of \cite[Ch. IX §3]{HelgasonDS} states that $\underline{a}(\bn)$ is of the form \[\exp \left(C_1\ln\left[(1+C_2|X|^2)^2+4C_2|Y|^2\right]H_0\right)\] where $H_0$ is a fixed element generating $\a$ and $C_1,C_2$ are constants depending on the dimensions of $\g_{-\alpha}$ and $\g_{-2\alpha}$. It is then straightforward to check the properness of the Iwasawa projection, studying the dependance in the coordinates $X$ and $Y$.

Finally, the higher rank situation boils down to the previous one by another classical reduction. Consider the set $\left\{\alpha_1,\ldots,\alpha_p\right\}$ of indivisible positive roots of $\a$ in $\g$. Let $\nb_{\alpha_i}=\g_{-\alpha_i}+\g_{-2\alpha_i}$ and $\N_{\alpha_i}=\exp (\bn_{\alpha_i})$. Then there exists a diffeomorphism \[\varphi:\prod_{i=1}^p\N_{\alpha_i}\longrightarrow\N\] such that \[\underline{a}\left(\varphi(\bn_1,\ldots,\bn_p)\right)=\underline{a}(\bn_1)\ldots\underline{a}(\bn_p),\] which concludes the proof.
\end{proof}

Theorem \ref{convergence} follows rather easily from the above lemma. Indeed, the formula defining $I_w$ is clearly $G$-equivariant, so that it is enough to establish the convergence of $\int_{\N}F(g_0 w \bn)\,d\bn$ for $F\in C_c(G/N)$ and any special $g_0\in G$. Chosing $g_0=w^{-1}$ and retaining the notation introduced above, the problem reduces to proving the convergence of \[\int_{\N}F\left(\underline{k}(\bn)\underline{a}(\bn)\right)\,d\bn\] since $F$ is $N$-invariant. Applying Lemma \ref{properness} concludes the proof.\hfill$\square$

\subsection{Expression in the open picture}

Although the most natural way to define $I_w$ is to consider a subspace of functions in $\E$, it may be useful to have explicit formulas in the other pictures at hand. In order to obtain a convenient expression, we first establish some formulas regarding the action of $G$ on the compact flag manifold $G/P$. The following lemma expresses the action $G\curvearrowright G/P$, composed with the `stereographic projection' of $G/P$ on the euclidean space $\N$\footnote{In the case of $G=\GR$, the group $\N$ is a real line on which the flag manifold $G/P\simeq\mathbb{S}^1$ surjects.} by means of the decomposition of almost all of $G$ according to $\N MAN$.

\begin{lemma}\label{transfo_w}
Let $g\in G$, $\bar{\nu}_0\in\N$ such that $g\bnu_0\in \N MAN$ and $\bnu\in\N\setminus\left\{1\right\}$. Set $\mu=w^2\in M$. Then,
\begin{enumerate}[(i)]
\item $\n(g^{-1}\n(g\bnu_0))=\bnu_0$
\item $\n(w\n(w\bnu))=c_{\mu}(\bnu)$
\item $\m(w\n(w\bnu)) = \mu\,\m(w\bnu)^{-1}$ 
\item $\ab(w\n(w\bnu)) = \ab(w\bnu)^{-1}$
\item $\lm(w\n(w\bnu)) = \mu\,\lm(w\bnu)^{-1}$
\end{enumerate}
\end{lemma}

\begin{proof}
Write $g\bnu_0$ with respect to $\N P$ as $g\bnu_0=\n({g\bnu_0}) p$. Then \[g^{-1}\n(g\bnu_0)=g^{-1}\n(g\bnu_0)p p^{-1}=g^{-1}g\bnu_0p^{-1},\] and $(i)$ is a consequence of unicity in the decomposition according to $\N MAN$. $(ii)$ follows from the remark that if $g\in G$ is such that $\n(g)$ exists and $m\in M$, then $\n(gm)=\n(g)$. Indeed, \[ \n(w\n(w\bnu))= \n(w\n(w^{-1}\mu\bnu))= \n(w\n(w^{-1}c_{\mu}(\bnu))),\] hence the result.
Finally, if $w\bnu=\bar{n}_0 p_0$ with $p_0=m_0a_0n_0$ and $w\bar{n}_0=\bar{n}_1p_1$ with $p_1=m_1a_1n_1$, then
\[\bar{n}_0 = w^{-1}\bar{n}_1p_1 = w\mu^{-1}\bar{n}_1p_1 = wc_{\mu^{-1}}(\bar{n}_1)\mu^{-1}p_1\] hence \[w\bnu = wc_{\mu^{-1}}(\bar{n}_1)\mu^{-1}p_1p_0 = wc_{\mu^{-1}}(\bar{n}_1)\mu^{-1}m_1m_0a_1a_0n',\]
for some $n'\in N$. Since $\m(w\n(w\bnu)) = m_1$ and $\ab(w\n(w\bnu)) = a_1$, formulas $(iii)$ and $(iv)$ follow, thus implying $(v)$.
\end{proof}

The action of the non-trivial Weyl element is sometimes called the \textit{inversion} of $\N$ (see \cite[Ch.2 §6]{HelgasonGASS}). Our next result shows how it transforms the measure on $\N$.

\begin{lemma}\label{transfo_w_int}
For $f\in L^1(\N)$, \[\int_{\N}f(\n(w\bnu))e^{-2\rho\log \ab(w\bnu)}\,d\bnu = \int_{\N}f(\bnu)\,d\bnu.\]
\end{lemma}

\begin{proof}
Since $w\in K$, the map $L_w:f\longmapsto f(w\cdot)$ preserves the $C^*(L)$-norm on $\E\simeq L^2(\N)\otimes C^*(L)$. Denote $P$ the map of Theorem \ref{Thm_Open} defined on elementary tensors of $C_c(\N)\otimes C_c(L)$ by \[P(f\otimes\varphi):\bar{n}ma\longmapsto e^{-\rho\log a}f(\bar{n})\varphi(ma).\] For any $x\in\N MA$ and $l_0\in L$, it is clear that $\n(xl_0)=\n(x)$ and $\lm(xl_0)=\lm(x)l_0$. Consequently, if $\bar{n}ma\in\N MA$, \[P(f\otimes\varphi)(w\bar{n}ma)=e^{-\rho\log a}e^{-\rho\log \ab(w\bar{n})}f(\n(w\bar{n}))\varphi(\lm(w\bar{n})ma),\]
and $L_w$ is given on $L^2(\N)\otimes C^*(L)$ by \[L_w(f\otimes\varphi)(\bar{n}ma)=e^{-\rho\log \ab(w\bar{n})}f(\n(w\bar{n}))\lambda_L(\lm(w\bar{n}))(\varphi)(ma).\] Since the $C^*(L)$-norm is given in the open picture by $|f\otimes\varphi|^2=\|f\|_2\varphi^*\varphi$ and preserved by $L_w$, denoting $f_w(\bar{n})=e^{-\rho\log \ab(w\bar{n})}f(\n(w\bar{n}))$, we get
\[\|f_w\|_2^2=\|f\|_2^2,\]
since the action of $\lambda_L(\lm(w\bar{n}))$ does not affect the norm. This proves the proposition for positive functions, and the result follows by linear combination.
\end{proof}

Let us finally describe the effect of conjugating $\bnu$ by elements of $L$ on the decomposition of $w^{-1}\bnu$ according to $\N MAN$.

\begin{lemma}\label{conj_L}
Let $l_0=m_0a_0\in L$ and $\bnu\in\N$. Then,
\begin{enumerate}[(i)]
\item $\m(w^{-1}c_{l_0}(\bnu)) = c_w(m_0)\m(w^{-1}\bnu)m_0^{-1}$
\item $\ab(w^{-1}c_{l_0}(\bnu)) = a_0^{-2}\ab(w^{-1}\bnu)$
\item $\lm(w^{-1}c_{l_0}(\bnu)) = c_w(m_0)\lm(w^{-1}\bnu)l_0^{-1}a_0^{-1}$
\end{enumerate}
\end{lemma}

\begin{proof}
$(iii)$ clearly follows from $(i)$ and $(ii)$. For those,
\begin{eqnarray*}
w^{-1}l_0^{-1}\bnu l_0 &=& c_{w^{-1}}(l_0^{-1})w^{-1}\bnu l_0 = c_{w^{-1}}(l_0^{-1})\bar{n}'\lm(w^{-1}\bnu)n'l_0\\
&=& \bar{n}''c_{w^{-1}}(l_0^{-1})\lm(w^{-1}\bnu)l_0n'',
\end{eqnarray*}
with $n',n''\in N$ and $\bn',\bn''\in\N$. $(i)$ follows from identifying the $M$ components. Taking into account the fact that $c_w$ acts on $A$ as the inverse map $a\mapsto a^{-1}$, identifying the $M$ components proves $(ii)$.
\end{proof}

We are now ready to write the standard intertwining integral in the open picture. To this purpose, we denote $\Iw$ the integral formula obtained by applying $I_w$ to elementary tensors of $L^2(\N)\otimes C^*(L)$. Namely, for $f\otimes\varphi$ in $C_c(\N)\otimes C_c(L)$ and $x_0 = \bn_0m_0a_0$, \[\Iw(f\otimes\varphi)(x_0) = e^{\rho\log (a_0)}\int_{\N}e^{-\rho\log \ab(x_0w\bar{\nu})}f(\n(x_0w\bnu))\varphi(\lm(x_0w\bnu))\,d\bnu.\]

\begin{notation}
The automorphism of $C^*(L)$ induced by the conjugation $c_w$ of $L$ is denoted $a\mapsto a^w$.
\end{notation}

\begin{proposition}\label{Iwopen}
Assuming that either side is defined, the equality \[\Iw(f\otimes\varphi)(x_0) = \int_{\N}f(\bn_0\bnu)\left[U_{\lm(w^{-1}\bar{\nu})}\varphi\right]^w(l_0)\,\frac{d\bnu}{|\bnu|}\] holds for $f\otimes\varphi\in C_c(\N)\otimes C_c(L)$ and $x_0 = \bn_0m_0a_0$.
\end{proposition}

\begin{proof}
The identities
\begin{gather*}
\m(x_0w\bnu) = m_0\m(w\bnu)\\
\ab(x_0w\bnu) = a_0\ab(w\bnu)\\
\n(x_0w\bnu) = \bn_0c_{l_0}(\n(w\bnu)) = \bn_0l_0\n(w\bnu)l_0^{-1}
\end{gather*}
are clear. They imply that \[\Iw(f\otimes\varphi)(x_0) = \int_{\N}e^{-\rho\log \ab(w\bnu)}f(\bn_0c_{l_0}(\n(w\bnu)))\varphi(l_0\lm(w\bnu))\,d\bnu.\]

The change of variables $\bnu\leftrightarrow\n(w\bnu)$ leads \textit{via} Lemma \ref{transfo_w_int} to the expression
\[\int_{\N} e^{-\rho\log\left(\ab(w\n(w\bnu))\ab(w\bnu)^2\right)}f\left[\bn_0c_{l_0}(\n(w\n(w\bnu)))\right]\varphi\left[l_0 \lm(w\n(w\bnu))\right]\,d\bnu,\] which simplifies to \[\int_{\N} e^{-\rho\log \ab(w\bnu)}f\left[\bn_0c_{l_0}(c_{\mu}(\bnu))\right]\varphi\left[l_0\mu \lm(w\bnu)^{-1}\right]\,d\bnu,\] using Lemma \ref{transfo_w}.

We computed the modular function $\text{mod}^{\N}c_l = e^{-2\rho\log a}$ for $l=ma\in L$ earlier. Besides, since $\mu=w^2\in M$, it follows that $wc_{\mu^{-1}}(\bnu) = w\mu^{-1}\bnu\mu = w^{-1}\bnu$ for $\bnu\in\N$ hence $\ab(wc_{\mu^{-1}}(\bnu)) = \ab(w^{-1}\bnu)$ and $\m(wc_{\mu^{-1}}(\bnu)) = \m(w^{-1}\bnu)\mu$, which leads to $\lm(wc_{\mu^{-1}}(\bnu)) = \lm(w^{-1}\bnu)\mu$. We deduce from the previous remarks and Lemma \ref{conj_L} that:

\begin{eqnarray*}
\Iw(f\otimes\varphi)(x_0) &=& \int_{\N}e^{-\rho\log \ab(wc_{\mu^{-1}}(\bnu))}f\left[\bn_0c_{l_0}(\bnu)\right]\varphi\left[l_0\mu \lm(wc_\mu{\mu^{-1}}(\bnu))^{-1}\right]\,d\bnu\\
&=&\int_{\N}e^{-\rho\log \ab(w^{-1}\bnu)}f\left[\bn_0c_{l_0}(\bnu)\right]\varphi\left[l_0\lm(w^{-1}\bnu)^{-1}\right]\,d\bnu\\
&\hspace{-2.5cm}=&\hspace{-1cm}\text{mod}^{\N}(c_{l_0^{-1}})\int_{\N}e^{-\rho\log \ab(w^{-1}c_{l_0^{-1}}(\bnu))}f(\bn_0\bnu)\varphi\left[l_0\lm(w^{-1}c_{l_0^{-1}}(\bnu))^{-1}\right]\,d\bnu\\
&\hspace{-2.5cm}=&\hspace{-1cm}\frac{e^{-2\rho\log a_0}}{\text{mod}^{\N}(c_{l_0})}\int_{\N}e^{-\rho\log \ab(w^{-1}\bnu)}f(\bn_0\bnu)\varphi\left[a_0^{-1}\lm(w^{-1}\bnu)^{-1}c_w^{-1}(m_0)\right]\,d\bnu
\end{eqnarray*}

Since $c_w(a_0)=a_0^{-1}$, it follows that \[\Iw(f\otimes\varphi)(x_0) = \int_{\N}f(\bn_0\bnu)\varphi\left[\lm(w^{-1}\bn^{-1}\bnu)^{-1}c_{w^{-1}}(l_0)\right]e^{-\rho\log \ab(w^{-1}\bnu)}\,d\bnu\]

Using the notations $\varphi\mapsto\varphi^w$ for the action of $w$ on $C^*(L)$ and $|\bn|$ for the norm function introduced before, we finally get the expected expression.
\end{proof}

It follows from Proposition \ref{Iwopen} that the standard intertwining integral may be written in the open picture as \[\Iw = \int_{\N}R_{\N}(\bnu)\otimes {U_{\lm(w^{-1}\bar{\nu})}}^w\,\frac{d\bnu}{|\bnu|},\] where $R_{\N}$ denotes the right regular representation of $\N$.

\begin{remark} The standard intertwining integral can also be expressed in the induced picture. More precisely, it is proved in \cite{These} that if $F\in\Eio$, this integral can be written \[I_w F(x)=\int_{\N}U_{\lm(w^{-1}\bn)}F(x\bn)\,\frac{d\bn}{|\bn|}\] for $x\in G/N$, and is well defined because of Theorem \ref{convergence} and the equivalence of the different pictures.
\end{remark}

Let us conclude this section by some general remarks on $C^*$-algebraic intertwining. Theorem \ref{convergence} replaces the meromorphic continuation in the theory of Knapp and Stein. The fact that the operators they obtain in this way are not unitary and need to be normalised has its counterpart in our framework. Namely, $I_w$ does not take its values in $\E$, as it was observed in \cite{These} already in the case of $\GR$. This phenomenon is the analogue of the existence of poles in the classical theory, where these singularities actually enclose all the data about reducibility in the $P$-series.

It is in fact possible to recover this information by studying the operator $I_w$, and more precisely what prevents it from extending to $\E$. In this picture, the `singularity' appears as a distribution $T_w$ on $L$, and it was observed in \cite{Notereducibility}, again for $\GR$, that studying the points in $\widehat{L}$ where the Fourier transform of $T_w$ vanishes also yields the parameters at which reducibility occurs.

Finally, let us sketch some perspectives about the normalisation of the intertwining operators. The procedure in Knapp-Stein theory consists in dividing the standard integral by a certain meromorphic function of the parameter in $\widehat{A}$ so that it becomes unitary. The normalising function is obtained by composing the standard operator by its adjoint. Similarly, our goal is to extract from $I_w$ a unitary operator on the Hilbert module $\E$, twisted by the automorphism of $C^*(L)$ induced by the Weyl element. However, since $I_w$ does not extend to $\E$, it is not possible to perform any kind of polar decomposition, so that one has to come up with an appropriate unitary an check that it actually normalises the standard integral. We came over this problem in \cite{These} by means of functionnal calculus on differential operators in some special cases.

\section{\texorpdfstring{Images of $C^*_r(G)$}{Images of C*r(G)}}\label{images}

In this final section, we argue how a $C^*$-algebraic theory of intertwiners should relate to the analysis of the reduced $C^*$-algebra of Lie groups. 
\begin{definition}
Two cuspidal parabolic subgroups are said to be associate if their Levi components are conjugate. The assocation class of a cuspidal parabolic subgroup $P$ is denoted by $[P]$.
\end{definition}

Results of Harish-Chandra \cite{HC56Char} and Lipsman \cite{Lipsman71} establish that the unitary equivalence classes of the $P$-series representations, for all possible $P$, can be partitioned according to the association classes of the inducing subgroups. This leads to the following decomposition of the reduced dual into a disjoint union \[\widehat{G}_r=\bigsqcup_{[P],\;P\,\text{cuspidal}}\widehat{G}_P\tag{\ensuremath{\divideontimes}}\label{dualG}\] where $\widehat{G}_P$ denotes the set of irreducible components in the $P$-series representations. This result, in relation with the Plancherel formula, describes the reduced dual as a measured space. In order to understand it as a noncommutative topological space, we need to take the reducibility phenomena into account. Those were dealt with by Knapp and Stein \cite{KS1,KS2} by means of their intertwining operators, implementing the intricate action of Weyl groups. The structure of the $C^*$-algebra was described in \cite{Valette85} for the cases where $\widehat{G}_r$ is Hausdorff, and in \cite{NoteWassermann} by means of the operators of Knapp and Stein.

We indicate here how the Hilbert modules $\E$ are expected to provide a good framework to analyse $C^*_r(G)$ with respect to association classes of cuspidal parabolic subgroups.

We no longer assume the real rank of $G$ to be $1$, but we consider $P=MAN$ minimal parabolic, which is a necessary and sufficient condition for $M$ to be compact.

\begin{proposition}
The left action of $C^*(G)$ on $\E$ induces a $*$-morphism \[C^*_r(G)\longrightarrow C_0(\widehat{M}\times\widehat{A},\mathcal{K})\]
\end{proposition}

\begin{proof}
Still denoting $\varepsilon_N:C^*(P)\twoheadrightarrow C^*(L)$ the canonical surjection, let $\tau$ be the map \[\C\rtimes G\simeq C^*(G)\longrightarrow C(G/P)\rtimes G \simeq\mathcal{K}(\EG)\] coming from Rieffel's formulation of the Imprimitivity Theorem, taking into account the compactness of $G/P$. Let $\tau_N=\tau\otimes_{\varepsilon_N}1$. Then \[\tau_N:C^*(G)\longrightarrow\mathcal{K}(\E).\]

Since $P$ is amenable, all its actions are. The Imprimitivity Theorem implies that the action of $G$ on $G/P$ is Morita-equivalent to the one of $P$ on $G/G$, hence amenable too (see \cite{CAD} and \cite{JeanetClaire}). It follows that $C(G/P)\rtimes G$ is isomorphic to $C(G/P)\rtimes_r G$, where $C^*_r(G)$ acts naturally. The situation is summed up in the following diagram


\[\xymatrix{\tau_N:\hspace{-1cm}&C^*(G)\ar@{->>}[d]_{\lambda_G}\ar[r]&C(G/P)\rtimes G\ar@{->}[d]_{\simeq}\ar[r]&\mathcal{K}(\E)\ar@{->}[d]^{\simeq}\\
&C_r^*(G)\ar@/^2.15pc/@{-->}[rr]\ar[r]&C(G/P)\rtimes_r G&C_0(\widehat{M}\times\widehat{A},\mathcal{K})\\
}\]

where the isomorphism in the last column is a consequence of the strong Morita equivalence between $\mathcal{K}(\E)$ and $C^*(L)$. Since $L=MA$ with $A$ abelian and $M$ compact, $C^*(L)$ decomposes into \[C_0(\widehat{A})\otimes\bigoplus_{\sigma\in\widehat{M}}\mathrm{End}\,(V_\sigma),\] where the $V_\sigma$ are finite-dimensional, hence the isomorphism \[\mathcal{K}(\E)\simeq C_0(\widehat{M}\times\widehat{A},\mathcal{K})\] by stable equivalence.
\end{proof}

It seems that the result extends to the cuspidal case with a $*$-morphism \[C^*_r(G)\longrightarrow C_0(\widehat{M}_d\times\widehat{A},\mathcal{K}).\] 
The purpose of our further work will be to describe the image of this morphism, using the unitary operators which arise by normalising $I_w$ (see \cite{These} for the cases of $\GR$ and $\GC$). Denoting $C^*_P(G)$ this image, the decomposition (\ref{dualG}) of $\widehat{G}_r$ is likely to translate at the level of $C^*_r(G)$ by \[C^*_r(G)\simeq\bigoplus_{[P],\;P\,\text{cuspidal}}C^*_P(G).\]

\bigskip

\begin{center}
\textbf{Acknowledgements}
\end{center}
The results presented here are part of the author's doctoral dissertation \cite{These}, written under patient and careful advisory of Pr. Pierre Julg in Orléans.

\bibliographystyle{amsplain}
\bibliography{biblio}
\end{document}